\documentclass[11pt]{amsart}

\usepackage{amssymb, amscd}
\usepackage{mathptmx}
\usepackage{lmodern}
\usepackage{dsfont}
\usepackage{xcolor}
\definecolor{cornsilk}{rgb}{1.0, 0.97, 0.86}

\usepackage{soul}
\setstcolor{magenta}
\numberwithin{equation}{section}

\usepackage{fullpage}

\DeclareSymbolFont{SY}{U}{psy}{m}{n}
\DeclareMathSymbol{\emptyset}{\mathord}{SY}{'306}

\newcommand{\overbar}[1]{\mkern 1.5mu\overline{\mkern-1.5mu#1\mkern-1.5mu}\mkern 1.5mu}

\usepackage{todonotes}

\theoremstyle{plain}

\newcommand{\thmlist}{
\renewcommand{\theenumi}{\alph{enumi}}
\renewcommand{\labelenumi}{(\theenumi)}}

\newtheorem*{bigT}{\sf Theorem}

\newtheorem{thm}{Theorem}[section]
\newtheorem{cor}[thm]{Corollary}
\newtheorem{lem}[thm]{Lemma}
\newtheorem{prop}[thm]{Proposition}

\newtheorem{defn}[thm]{Definition}

\newtheorem{rem}[thm]{Remark}
\newtheorem{ex}[thm]{Example}

\newcounter{defcounter}
\makeatletter
\def\moverlay{\mathpalette\mov@rlay}
\def\mov@rlay#1#2{\leavevmode\vtop{%
   \baselineskip\z@skip \lineskiplimit-\maxdimen
   \ialign{\hfil$\m@th#1##$\hfil\cr#2\crcr}}}
\newcommand{\charfusion}[3][\mathord]{
    #1{\ifx#1\mathop\vphantom{#2}\fi
        \mathpalette\mov@rlay{#2\cr#3}
      }
    \ifx#1\mathop\expandafter\displaylimits\fi}
\makeatother
\newcommand{\cupdot}{\charfusion[\mathbin]{\cup}{\cdot}}
\newcommand{\bigcupdot}{\charfusion[\mathop]{\bigcup}{\cdot}}

\begin{document}

\title[Mackey Imprimitivity and homogeneous normal operators]{Mackey Imprimitivity and commuting tuples of homogeneous normal operators}



\author[G. Misra]{Gadadhar Misra}

\address[G. Misra]{Indian Statistical Institute, Bangalore 560 059 AND Indian Institute of Technology, Gandhinagar 382 055}
\email[G. Misra]{gm@isibang.ac.in}
\thanks{The first author would like to thank his colleagues at the department of mathematics, IIT Gandhinagar, for providing a congenial and friendly atmosphere for the conduct of his research.}
\author{E. K. Narayanan}
\address[E. K. Narayanan]{Indian Institute of Science, Bangalore 560 012}
\email[E. K. Narayanan]{naru@iisc.ac.in}

\author{Cherian Varughese}
\address[C. Varughese]{Renaissance Communications, Bangalore - 560 058}
\email[C. Varughese]{cherian@rcpl.com}
\thanks{The authors thank Jean-Louis Clerc and Adam K\'{o}ranyi for suggesting that it might be helpful to investigate the case of the product domains first.}



\date{}

\dedicatory{Dedicated to the memory of K. R. Parthasarathy}
\begin{abstract}
In this semi-expository article, we investigate the relationship between the imprimitivity introduced by Mackey several decades ago and commuting $d$- tuples of homogeneous normal operators. The Hahn-Hellinger theorem gives a canonical decomposition of a $*$- algebra representation $\rho$ of $C_0(\mathbb{S})$ (where $\mathbb S$ is a locally compact Hausdorff space) into a direct sum. If there is a group $G$ acting transitively on $\mathbb{S}$ and is adapted to the $*$- representation $\rho$ via a unitary representation $U$ of the group $G$, in other words, if there is an imprimitivity, then the Hahn-Hellinger decomposition reduces to just one component, and the group representation $U$ becomes an induced representation, which is Mackey's imprimitivity theorem. We consider the case where a compact topological space $S\subset \mathbb {C}^d$ decomposes into finitely many $G$- orbits. In such cases, the imprimitivity based on $S$ admits a decomposition as a direct sum of imprimitivities based on these orbits. This decomposition leads to a correspondence with homogeneous normal tuples whose joint spectrum is precisely the closure of $G$- orbits.

\end{abstract}

\maketitle

\section{Introduction}
Let $G$ be a locally compact second countable group and $\mathbb{S}$ be a locally compact Hausdorff transitive $G$- space. Given a $*$- homomorphism $\rho$ from $C_0(\mathbb{S})$, the algebra of continuous functions  vanishing at $\infty$ on $\mathbb{S}$ to $\mathcal{L}(\mathcal{H})$, the algebra of bounded linear operators on a complex separable Hilbert space $\mathcal{H}$ and a unitary representation $U$ of the group $G$ on the same Hilbert space $\mathcal{H}$, the imprimitivity is the relationship 
\[U(g) \rho(f) U(g)^* = \rho(g\cdot f), \,\, g\in G,\,\, f\in C_0(\mathbb{S}),\]
where $g\cdot f$ is the function: $(g\cdot f)(s) = f(g^{-1}\cdot s)$, $s\in \mathbb{S}$. 

In the transitive case, $\mathbb{S}$ can be taken to be the space of cosets $G/H$ for some closed subgroup $H$ of $G$. There is a unique quasi-invariant measure (modulo equivalance) $\mu$ on $G/H$, that is, setting $g_*\mu(A):=\mu(g(A))$ for any Borel subset $A\subset \mathbb{S}$, we require that $g_*\mu$ be  mutually absolutely continuous with respect to $\mu$ for every $g\in G$. 

 A commuting $d$- tuple $\boldsymbol{N}= (N_1, \ldots ,N_d)$ of normal operators acting on a complex separable Hilbert space $\mathcal{H}$ is said to be homogeneous with respect to a group $G$ if the joint spectrum $\sigma_{\boldsymbol{N}}\subset \mathbb{C}^d$ is a $G$- space and there is a unitary representation $U$ of $G$ on $\mathcal{H}$ such that 
 \begin{align*} 
 U(g)^* \boldsymbol{N} U(g) &:= ( U(g)^* N_1U(g) , \ldots ,U(g)^* N_d U(g) ) \\ 
 &= (g_1(\boldsymbol{N}) , \ldots , g_d(\boldsymbol{N}))
 := g (\boldsymbol{N}), 
 \end{align*}
 where $g_i, 1\leqslant i \leqslant d$, are the coordinate functions of the action of $G$ on $\sigma_{\boldsymbol{N}}$, namely, \[g\cdot s:= (g_1(s), \ldots , g_d(s)).\]  

We explore the relationship of imprimitivities with commuting $d$- tuples of homogeneous normal operators. For this, we assume throughout that the $G$- space $\mathbb{S}$ is a subset of $\mathbb{C}^d$. Somewhat surprisingly, such a relationship was hinted in \cite{BM}. 

The imprimitivity theorem of Mackey has two parts, namely, that any transitive imprimitivity $(\mathbb{S}, U, \rho)$ is equivalent to a canonical imprimitivity, where $\rho(f)$ for $f\in C_0(\mathbb{S})$ is defined to be the operator $M_f$ of multiplication by $f$ on $L^2(\mathbb{S},\mu, \mathcal{H}_n)$ and $U$ is a multiplier representation on $L^2(\mathbb{S},\mu, \mathcal{H}_n)$, that is, 
\[\big (U(g) h \big )(s) = c(g, s) (g\cdot h)(s),\,\, h\in L^2(\mathbb{S},\mu, \mathcal{H}_n),\,\, g\in G,\] 
where $c:G\times \mathbb{S} \to \mathcal{U}(\mathcal{H}_n)$ is a Borel map taking values in the group of unitary operators $\mathcal{U}(\mathcal{H}_n)$ of the Hilbert space $\mathcal{H}_n$ of dimension $n$. For $U$ to be a homomorphism, the function $c$ must be a cocycle. The second part of the imprimitivity theorem asserts that such a multiplier representation is induced from a unitary representation of the subgroup $H$ acting on the Hilbert space $\mathcal{H}_n$. 

It is evident that the $d$- tuple of multiplication by coordinate functions $(M_1, \ldots , M_d)$ acting on $L^2(\mathbb{S},\mu, \mathcal{H}_n)$ is homogeneous. 
We prove that any $d$- tuple $\boldsymbol{N}$ of commuting homogeneous normal operators such that $\sigma(\boldsymbol{N})$ is a $G$- space is actually equivalent to a direct sum of several transitive 
imprimitivities that are taken to be of the canonical form without loss of generality. 

Since in the simplest of examples like that of the closed unit disc, where the group of M\"{o}bius transformations act holomorphically but the action is not transitive, we set out to study some class of imprimitivities that do not come from a transitive action. (This example was first mentioned in \cite{survey}.) We prove that, under mild hypothesis, these are direct sums of transitive imprimitivities. For a precise statement, see Corollary \ref{maincor}. We then apply this result to several examples and list all the homogeneous $d$- tuples of normal operators modulo unitary equivalence. 

In the Remark on page 225 of \cite{VSV}, it is stated that ``Theorem 6.12 gives a complete analysis of the transitive homogeneous systems of imprimitivity based on X. \ldots In view of this it has not been possible to carry out the analysis of ergodic systems of imprimitivity further than that of theorem 6.11.'' In this paper, we  identify a class of imprimtivities where neither the $G$ action is transitive nor is the spectral measure of uniform multiplicity and yet the imprimitivities are described completely modulo unitary equivalence, see Theorem \ref{multR} and Corollary \ref{maincor}. We discuss several examples where the theorem applies.

In the concluding section, we state two problems that we think might be of some interest to researchers in operator theory as well as representation theory. 

\section{Spectral theorem and the Hahn-Hellinger decomposition} 
We begin by describing $*$- representations of a commutative $C^*$-algebra. In the case of a non-unital $C^*$- algebra, a $*$- representation $\rho: C_0(\mathbb{S}) \to \mathcal L(\mathcal H)$ is said to be \emph{non-degenerate} if there is no nonzero $h \in \mathcal{H}$ such that $\rho(f) h =0$ for all $f \in C_0(\mathbb{S})$. Let $\mathcal{B}$ be the $\sigma$- algebra of all Borel sets in $\mathbb S$ and $\mathcal{P}(\mathcal{H})$ be the set of projections on a separable complex Hilbert space $\mathcal{H}$. 

\begin{defn}
     A spectral measure defined on $\mathbb{S}$ is a projection valued map $P: \mathcal{B} \rightarrow \mathcal{P}(\mathcal{H})$ such that $P(\mathbb{S})=I$ and $P(\cupdot E_k) = \sum_{k=1}^\infty P(E_k)$ for any disjoint collection of sets  $E_k$, $k=1,2, \ldots ,$ in $\mathcal{B}$, where the convergence is in the strong operator topology. The spectrum $\sigma(P)$ of a spectral measure is the complement in $\mathbb{S}$ of the union of all those open subsets $E$ of $\mathbb{S}$ such that $P(E) = 0$. 
    A spectral measure is said to be compact if $\sigma(P)$ is compact. If $P$ is a spectral measure for $(\mathbb{S}, \mathcal B)$ and $x, y \in \mathcal{H}$, then 
\[
P_{x, y}(E) \equiv\langle P(S) x, y\rangle,\, x,y \in \mathcal H;\, S\in \mathcal B,
\]
defines a countably additive measure on $\mathbb{S}$. A projection-valued measure $P$ on $(\mathbb{S}, \mathcal B)$ is called regular if each of the measures $P_{x, y}$ is regular.
\end{defn}


\begin{thm}[Corollary 1.55, \cite{GF}] 
    Suppose that $\mathbb{S}$ is a locally compact Hausdorff space, and $\rho$ is a nondegenerate $*$- representation of $C_0(\mathbb{S})$ on $\mathcal{H}$. Then there is a unique regular projection-valued measure $P$ on $\mathbb{S}$ such that $\rho(f)=\int f d P$ for all $f \in C_0(\mathbb{S})$.
\end{thm}
If $S$ is compact, then the $*$- representation $\rho$ of $C(S)$ extends to the algebra $B(S)$ of bounded Borel functions on $S$. 
\begin{thm}[Chapter IX, Theorem 1.14, \cite{JBC}] \label{thm:1.6} If $S$ is a compact topological space and $\rho: C(S) \rightarrow \mathcal{L}(\mathcal{H})$ is a $*$- representation, there is a unique spectral measure $P$ defined on the Borel subsets of $S$ such that for all $x$ and $y$ in $\mathcal{H}$, $P_{x, y}$ is a regular measure and 
$\rho(f)=\int f d P$ for all $f\in C(S)$.
\end{thm}
We reproduce a key idea of the proof of Theorem \ref{thm:1.6} above from \cite{JBC}.
If $x, y \in \mathcal{H}$, then $\Lambda_{x,y}(f):= \langle \rho(f) x , y\rangle $ is a linear functional on $C(S)$ with $\|\Lambda_{x,y} \| \leqslant  \|x\|\|y\|$. Hence there is a unique regular Borel measure $\mu_{x, y}$ such that
\[
\langle\rho(f) x, y\rangle=\int f d \mu_{x, y}
\]
for all $f$ in $C(S)$. 
Let $\phi$ be a bounded Borel function on $S$ and define $[x, y]=\int \phi d \mu_{x, y}$. Evidently, $[\cdot, \cdot]$ is a sesquilinear form and $|[x, y]| \leqslant\|\phi\|\|x\|\|y\|$. Hence there is a unique bounded operator $A_\phi$ such that $[x, y]=\langle A_\phi x, y\rangle$ and $\|A_\phi\| \leqslant\|\phi\|$. Define   
$\tilde{\rho}: B(S) \rightarrow \mathcal{L}(\mathcal{H})$ by setting $\tilde{\rho}(\phi) = A_\phi$. Thus,  
\[
\langle\tilde{\rho}(\phi) x, y\rangle=\int \phi d \mu_{x, y}.
\]
It is shown in the proof of Theorem 1.14 of \cite{JBC} that (a) $\tilde{\rho}: B(S) \rightarrow \mathcal{L}(\mathcal{H})$ is a $*$- representation and $\tilde{\rho}_{\mid\,C(S)} =\rho$ and (b) $P(E):= \tilde{\rho}\left(\chi_{E}\right)$, $E\in \mathcal B$, is a spectral measure, where $\chi_E$ is the characteristic function of $E$. 
\begin{rem}
Any $*$- representation $\rho$ of $C(S)$, $S$ compact, not only (uniquely) defines a spectral measure, it is actually the restriction of a $*$- representation of the algebra of bounded Borel functions on $S$. 
\end{rem}
Finally, using the Hahn-Hellinger theorem, one may write down any $*$- representation of $C_0(\mathbb{S})$ in a canonical form, see \cite[Theorem 3.7]{VSS}. Let $\mu$ be a $\sigma$-finite Borel measure on $\mathbb{S}$, and $\mathcal{H}_n$ denote an $n$ dimensional Hilbert space, $1 \leqslant n \leqslant \aleph_0$. The map $\pi_\mu^n$ defines a $*$- representation of $C(\mathbb{S})$ by 
setting 
\[(\pi_\mu^n(f) h)(x)=f(x) h(x),\,\, f\in C(\mathbb{S}),\,\,h\in L^2\left(\mathbb{S}, \mu ; \mathcal{H}_n\right).\]
\begin{thm} \label{Hahn}
Let $\pi: C_0(\mathbb{S}) \rightarrow \mathcal{L}(\mathcal{H})$ be a $*$- representation of $C_0(\mathbb{S})$ on a separable Hilbert space $\mathcal{H}$.
\begin{enumerate} 
\item[(a)] Then, there exists a Borel measure $\mu$ on $\mathbb S$, and a sequence $\left\{E_n\right.$ : $\left.1 \leq n \leq \aleph_0\right\}$ of pairwise disjoint Borel sets in $S$, and Hilbert spaces $\mathcal{H}_n, 1 \leqslant n \leqslant \aleph_0$, with $\operatorname{dim} \mathcal{H}_n=n$, such that $\mu\left(\mathbb{S} \setminus (\cupdot_n E_n)\right)=0$, and
\[\pi \cong \bigoplus_{1 \leq n \leq \aleph_0}
\pi_{\left .\mu\right|_{E_n}}^n.\]
\item[(b)] Further, if $\pi^{\prime}$ is another representation of $C_0(\mathbb{S})$, and if $\mu^{\prime}$ is a measure and if $\left\{E_n^{\prime}: 1 \leq n \leq \aleph_0\right\}$ is a sequence of pairwise disjoint Borel sets such that 
\[\pi^{\prime} \cong \bigoplus_{1 \leq n \leq \aleph_0} \pi_{\left.\mu^{\prime}\right|_{E_n^{\prime}}}^n,\] and if $\pi \cong \pi^{\prime}$, then $\mu \cong \mu^{\prime}$ and $\mu\left(E_n \triangle E_n^{\prime}\right)=0$ for all $n$, where $\triangle$ denotes ``symmetric difference'': $A \Delta B=(A\cap B^c) \cup(B\cap A^c)$.
\end{enumerate}
\end{thm}

We need to discuss the relationship of the spectrum of the spectral measure $P$ defined on a locally compact space $\mathbb{S}$ with the (joint) spectrum of $\boldsymbol N$ (commuting tuple of normal operators) defined by $P$ as in Theorem \ref{thm:1.6}. However, there are several different ways in which one may define the spectrum of $\boldsymbol{N}$, and these are discussed below. 

Let $\boldsymbol{N} = (N_1, N_2, \cdots, N_d)$ be a commuting tuple of normal operators acting on a Hilbert space $\mathcal H$. Some of what we are going to say applies to any commuting tuple of bounded operators but we will restrict ourselves to normal operators. In this case, some of the arguments are much simpler. 

First, the {\em left spectrum} $\sigma_{\text{left}}(\boldsymbol{N})$ is defined to be those $\boldsymbol{\lambda}:=(\lambda_1, \ldots, , \lambda_d)\in \mathbb C$ such that the left ideal generated by $\left\{\left(N_1-\lambda_1\right),\left(N_2-\lambda_2\right), \cdots,\left(N_d-\lambda_d\right)\right\}$ in the unital (commutative) $C^*$- algebra $C^*(\boldsymbol{N})$ generated by $\boldsymbol{N}$ is proper. The right spectrum $\sigma_{\text{right}}(\boldsymbol{N})$ is defined similarly. 

From the Gelfand theory, see \cite[Theorem B, p. 320]{GFS}, it follows that $C^*(\boldsymbol{N})$ is $*$- isomorphic to the algebra of continuous functions $C(\mathcal{M})$ on the space $\mathcal{M}$ of multiplicative linear functionals of $C^*(\boldsymbol{N})$. Then we have  \begin{equation}
\sigma_{\text{left}}(\boldsymbol{N}) = \sigma_{\text{right}}(\boldsymbol{N}) = \left\{\left(\ell\left(N_1\right), \ldots, \ell\left(N_d\right)\right): \ell \in \mathcal M\right\} \subset \mathbb C^d, 
\end{equation}
see \cite[Lemma 2.1]{Harte} and also the remark following \cite[Corollary 4]{Bunce}. 

\begin{rem} \label{spectralperm}
It may appear that the definition of the spectrum, say, the left spectrum,  depends on the algebra $C^*(\boldsymbol{N})$ since $(\lambda_1, \ldots , \lambda_d) \not \in \sigma_{\text{left}}(\boldsymbol{N})$, only if we can find $S_1, \ldots, S_d\in C^*(\boldsymbol{N})$ such that $\sum_{j=1}^d S_j(N_j-\lambda_j) = I$. However, if $C^*(\boldsymbol{N}) \subset \mathcal A$ such that the identity of $\mathcal A$ also serves as the identity in $C^*(\boldsymbol{N})$, then the spectrum $\sigma_{\text{left}}(\boldsymbol{N}) = \sigma_{\mathcal A}(\boldsymbol{N})$, see \cite[Corollary 8]{curto}. 
A particular case that is important in what follows occurs by choosing  $\mathcal A = C^*(\boldsymbol{N})^{\prime\!\prime}$.  
Since we have $\sigma_{\text{left}}(\boldsymbol{N})= \sigma_{\text{right}}(\boldsymbol{N})$, we drop the subscripts ``left'' and ``right'' and let $\sigma(\boldsymbol{N})$ denote the spectrum of $\boldsymbol{N}$. This is known as the Harte spectrum.  
\end{rem}

A second notion of the spectrum is the {\em joint approximate spectrum} $\sigma_a(\boldsymbol{N})$: The $d$- tuple $\left(\lambda_1, \cdots, \lambda_d\right)$ of complex numbers is said to be in $\sigma_a({\boldsymbol{N}})$ if there exists a sequence  $x_n$ of unit vectors in $\mathcal{H}$  such that $\left(N_j-\lambda_j\right) x_n \rightarrow 0$, $1\leqslant j \leqslant d$. The proof of the following lemma is in \cite{CobSec}.
We provide a proof for the sake of completeness. 
\begin{lem} \label{approxspec}
    If $\boldsymbol{N}$ is a $d$- tuple of commuting normal operators, then $\sigma_{\text{left}}(\boldsymbol{N}) =  \sigma_a(\boldsymbol{N})$. 
\end{lem}
\begin{proof} Let $\mathcal{I} =\left\{\sum_{j=1}^d S_j\left(N_j-\lambda_j\right): S_j \in \mathcal{L}(\mathcal{H})\right\}$. 

To prove that $\sigma_a(\boldsymbol{N}) \subseteq \sigma(\boldsymbol{N})$, consider $\left(\lambda_1, \cdots, \lambda_d\right) \in \sigma_a(\boldsymbol{N})$, i.e., $\left(N_j-\lambda_j\right) x_n \rightarrow 0$, for every $x\in \mathcal H$ with $\left\|x_n\right\|=1$. If $\mathcal{I}$ is not proper, then $I \in \mathcal{I}$.
Hence, there exists  $S_j \in \mathcal{L}(\mathcal{H})$, $j=1,2, \cdots, d$, such that $\sum_{j=1}^d S_j\left(N_j-\lambda_j\right)=I$, i.e., $\sum_{j=1}^d S_j\left(N_j-\lambda_j\right) x_n=x_n$
This contradicts the fact that $\left(N_j-\lambda_j\right) x_n \rightarrow 0$ since $\left\|x_n\right\|=1$. Hence $\mathcal{I}$ is proper and $\left(\lambda_1,  \cdots, \lambda_d\right) \in \sigma(\boldsymbol{N})$.

To prove that $\sigma(\boldsymbol{N}) \subseteq \sigma_a(\boldsymbol{N})$, consider $\left(\lambda_1, , \cdots, \lambda_k\right) \in \sigma_l(\boldsymbol{N})$. Then $\mathcal I$ is proper and, therefore, consists of non-invertible elements.
Hence $\sum_{j=1}^d \left(N_j-\lambda_j\right)^*\left(N_j-\lambda_j\right) \in \mathcal{I}$ is non-invertible, that is, $0$ belongs to the spectrum of the operator $\sum_{j=1}^k\left(N_j-\lambda_j\right)^*\left(N_j-\lambda_j\right)$. 
Since the operator is self-adjoint, all its spectral values are approximate eigenvalues. Hence, there exists $\left\{x_n \in \mathcal{H}:\left\|x_n\right\|=1, n\in \mathbb N\right\}$ such that $\sum_{j=1}^k\left(N_j-\lambda_j\right)^*\left(N_j-\lambda_j\right) x_n \rightarrow 0$, $n \to \infty$.
It follows that $\sum_{j=1}^k\left\|\left(N_j-\lambda_j\right) x_n\right\|^2 \rightarrow 0$ and $\left(N_j-\lambda_j\right) x_n \rightarrow 0$, $n \to \infty$, for every $j$. Hence $\left(\lambda_1,  \cdots, \lambda_d\right)$ is in $\sigma_a(\boldsymbol{N})$.
\end{proof}
Recall that the spectral measure $P$ defines a commuting $d$- tuple of normal operators as described below. This is the spectral theorem for a $d$- tuple of commuting operators.
\begin{thm}[Theorem 3, \cite{Hastings}]
    Let $\boldsymbol{N}=\left(N_1, \ldots, N_d\right)$ be a $d$- tuple of pairwise commuting normal operators on acting on a Hilbert space $\mathcal H$ and let $W^*(\boldsymbol{N}):=\left\{N_1, \ldots, N_d\right\}^{\prime \prime}$ be the commuting von Neumann algebra consisting of all those operators that doubly commute with $\left\{N_1, \ldots, N_d\right\}$. Then there exists a projection-valued spectral measure $P$ with  $\operatorname{supp}(P)=\sigma(\boldsymbol{N})$ such that
\[
N_i=\int z_i d P(z), \quad 1 \leq i \leq d,
\]
where $\sigma(\boldsymbol{N})$ is computed relative to $W^*(\boldsymbol{N})$
\end{thm}
From Remark \ref{spectralperm}, we conclude that $\sigma(\boldsymbol{N})$ is unambiguously defined irrespective of whether we define it relative to the algebra $W^*(\boldsymbol{N})$ or the algebra $C^*(\boldsymbol{N})$. The claim that $\operatorname{supp}(P)=\sigma(\boldsymbol{N})$ follows from \cite[Theorem 2]{CobSec} as pointed out in \cite{Hastings}. We verify this equality closely following the proof of Theorem 2 from \cite[p. 64]{PRH}.

\begin{lem} Assume that $P$ is a compact spectral measure on a locally compact space $\mathbb {S}\subset \mathbb C^d$. Then the support of the spectral measure $P$ and the spectrum of the commuting tuple $\boldsymbol N$ of operators, where $N_i=\int z_i d P(z), \, 1 \leq i \leq d,$ are equal. 
\end{lem}
\begin{proof} From Lemma \ref{approxspec}, we have that $\sigma(\boldsymbol{N})=\sigma_a(\boldsymbol{N})$. Thus, it is enough to prove  $\operatorname{supp}(P)=\sigma_a(\boldsymbol{N})$.  

Pick $\boldsymbol{\lambda}^0\not \in \operatorname{supp}(P)$. Then there exists an open neighbourhood $U$ of $\boldsymbol{\lambda}^0$ such that $P(U) = 0$. Let $U^\prime$  be the compliment of $U$ and $\delta=\operatorname{dist}(\boldsymbol{\lambda}^0, U^\prime)$. Now, we have 
\begin{align*}
    \sum_{j=1}^d \big \|(N_j - \lambda_j^0 I)x \big \|^2 & =  \sum_{j=1}^d \langle (N_j - \lambda_j^0)^* ( N_j - \lambda_j^0 ) x, x\rangle \\
    &= \sum_{j=1}^d \int \big (\overline{(\lambda_j - \lambda_j^0)}(\lambda_j - \lambda_j^0) \big ) d \langle P(\boldsymbol{\lambda}) x , x\rangle 
\end{align*}
for all $x\in \mathcal H$. Since $P(U) = 0$, it follows that 
\begin{align*}
\big (  \sum_{j=1}^d \big \|(N_j - \lambda_j^0 I)x \big \| \big )^2 & \geqslant  \sum_{j=1}^d \big \|(N_j - \lambda_j^0 I)x \big \|^2 \\
& = \int_{U^\prime} \sum_{j=1}^d |\lambda_j - \lambda_j^0|^2 \langle P(\boldsymbol{\lambda}) x , x\rangle\\
&\geqslant  \delta^2 \|x\|^2 
\end{align*}
for all $x\in \mathcal H$. Consequently, $\displaystyle\sum_{j=1}^d\left\|\left(N_j-\lambda_j^0\right) x\right\| \geqslant \delta \|x\|$ and 
therefore $\boldsymbol{\lambda}^0\not \in \sigma_a(\boldsymbol{N})$.

Conversely, if $\boldsymbol{\lambda}^0\in \operatorname{supp}(P)$, then $P(U) \ne 0$ for every open set containing $\boldsymbol{\lambda}^0$. Hence if $\delta$ is any positive number such that $B=\{\boldsymbol{\lambda}: |\boldsymbol{\lambda} - \boldsymbol{\lambda}^0| < \delta \}$, then there is a non-zero vector $x$ in the range of $P(B)$. Now, arguing as in the first half of the proof, we conclude that 
\[\big (  \sum_{j=1}^d \big \|(N_j - \lambda_j^0 I)x \big \| \big )^2 =\int_{B} \sum_{j=1}^d |\lambda_j - \lambda_j^0|^2 \langle P(\boldsymbol{\lambda}) x , x\rangle \leqslant \delta^2 \|x\|^2. \]
It follows that $\boldsymbol{\lambda}^0\in \sigma_a(\boldsymbol{N})$. 
\end{proof}
Adapting the terminology from \cite{PRH}, let us say that a spectral measure defined on a subset $\mathbb{S}$ of $\mathbb C^d$ is a {\em complex} spectral measure.  
It is proved in \cite[Theorem 1, p. 63]{PRH} that a complex spectral measure is regular for $d=1$. The same proof works with the obvious modifications for any $d\in \mathbb N$. We record below this property of a complex spectral measure without proof.  
\begin{prop}
    Every complex spectral measure is regular. 
\end{prop}

\section{Multiplier representations and Induced representations}\label{section:Indrepn}

This section is devoted to establishing a correspondence between multiplier representations and induced representations. We start with some generalities.

Let $G$ be a locally compact second countable group and $\mathbb{S}$ be a locally compact $G$- space, that is, there is a map $\alpha:G\times \mathbb{S} \to \mathbb{S}$, such that for a fixed $g\in G$,  $s\to \alpha_g $, $\alpha_g(s):= \alpha(g,s)$, is a bijective and continuous map of $\mathbb{S}$, moreover, $g\to \alpha_g$ is a homomorphism. We let $g\cdot s:=\alpha(g,s)$, $g\in G$ and $s\in \mathbb{S}$.  The action of $G$ on $\mathbb{S}$ is said to be \textit{transitive} if for every pair  $s_1,s_2$ in $\mathbb{S}$, there is a $g\in G$ such that $g \cdot s_1 = s_2$. Let $H\subseteq G$ be a closed subgroup and let $\mathbb{S}:=G / H$ be the space of cosets: $\{g H\mid g \in G\}$. Equipped with the action of $G$ by left multiplication: $g^\prime (g H):= (g^\prime g) H$, $g^\prime, g\in G$, the coset space $\mathbb{S}$ is a transitive $G$-space. On the other hand, any transitive $G$-space must be of this form, see \cite[Proposition 2.46]{GF}. Following \cite{GF}, let us say that $\mathbb{S}$ is a \textit{homogeneous} space if it is homeomorphic to $G / H$. In this case, we identify $\mathbb{S}$ with $G/H$ and define $(g\cdot f)(s) = f (g^{-1} \cdot s)$ for any function defined on $\mathbb{S}$. 

Let $(\mathbb{S}, \mathcal B)$ be the Borel measurable space, and note that each $g\in G$ defines a continuous map on $\mathbb{S}$ by our assumption.  Given a $\sigma$-finite measure $\mu$ on $\mathbb{S}$, define the {\it push-forward}  $g_*\mu$  of the measure $\mu$ by the requirement
\[(g_*\mu )  (A) := \mu \big (g\cdot A \big ),\,\, g\cdot A:= \{g^{-1}\cdot s \mid s\in A \}, A\in \mathcal B.\]
The measure $\mu$ on $G$ is said to be invariant if $g_*\mu = \mu$ and \textit{quasi-invariant} if 
$g_*\mu$ is equivalent (mutually absolutely continuous) to $\mu$ for all $g\in G$.  
\begin{rem}\label{cross-quasi}
If $G$ is second countable, then we have the following two very useful tools at our disposal. 
\begin{enumerate}
\item There is a Borel cross-section $p:G/H \to G$, that is, a Borel subset $B \subset G$ that meets each coset of $H$ in exactly one point, see \cite[Corollary, Theorem 3.41]{WBA}, and \cite[Lemma 1.1]{MA52}. Thus, each $g \in G$ can be written uniquely as $g =g_1 g_0$ with $g_0 \in H$ and $g_1 \in B$. 
\item There is a quasi-invariant measure uniquely determined modulo mutual absolute equivalence on $\mathbb S$, see \cite[Theorem 2.58]{GF}.
\end{enumerate}
\end{rem}

The inducing construction is a way of producing unitary representations of a locally compact group $G$ from unitary representations of a closed subgroup $H$ of $G.$ In this section we describe two pictures of {\textit{induced representations}}.

\begin{enumerate}
\item[(i)] The first is that of a multiplier representation, which is more suitable for function theory/operator theory applications.

\item[(ii)] Second one is the standard representation theoretic construction of a $G-$action on a space of vector valued functions on $G$ that transform according to the given representation of $H.$

\end{enumerate}

\subsection{Multiplier representations}

Let $\mathbb S$ be a second countable Hausdorff space on which a Lie group $G$ acts transitively. For $x \in \mathbb S$ and $g \in G,$ $g \cdot x$ denotes the action of $g$ on the point $x.$ If $x_0 \in \mathbb S,$ the stabilizer of $x_0$ is the closed subgroup $H$ of $G$ defined by $$H = \{h \in G:~h \cdot x_0 = x_0\}.$$ 
We fix $x_0\in \mathbb{S}$ in what follows. 
Since the action is transitive, we can identify the space $\mathbb S$ with the homogeneous space $G/H.$ Recall that $G/H$ admits a quasi-invariant measure, which we denote by $\mu.$


\begin{ex}
Let $\mathbb B^n = \{z \in \mathbb C^n:~|z| < 1 \}$ denote the Euclidean unit ball in $\mathbb C^n.$ The group of bi-holomorphic automorphisms on $\mathbb B^n,$ denoted by {\rm M\"ob}$(\mathbb B^n)$ can be described as follows:

For each $a \in \mathbb B^,$ let $P_a$ be the orthogonal projection of $\mathbb C^n$ onto the subspace generated by $a.$ So, $P_0 = 0$ and $$P_a z =
\frac{\langle z, a \rangle}{\langle a, a \rangle} a, \quad a \neq 0.$$ Let $Q_a = I-P_a$ be the orthogonal projection onto the orthogonal complement of $a.$
Define
\begin{equation}\label{mob-unit-ball}
\varphi_a(z) = \frac{a-P_a z - (1-|a|^2)^{1/2} Q_a z}{1 - \langle z, a \rangle}.
\end{equation}
The bi-holomorphic map $\varphi_a$ is an involution which interchanges the points $0$ and $a.$ Then, the group M\"ob$(\mathbb B^n)$ is given by 
\[ {\text{\rm M\"ob}}(\mathbb B^n) = \{U\varphi_a: a \in \mathbb B^n , U \in U(n)\}.\] 

The unit ball $\mathbb B^n$ can be identified with $G/K$ where $G = \text{M\"ob} (\mathbb B^n)$ and $K = SU(n).$ In this case, it is well known that $G/K$ carries a measure that is $G$-invariant.
\end{ex}

Now we describe the multiplier representations. Let $m : G \times \mathbb S \to \mathcal{U}(V)$ be a Borel function, where $\mathcal{U}(V)$ is the space of unitary operators on a complex separable Hilbert space $V.$ 

Define $$T_g f(x) = \left(\frac{d (g_*\mu)}{d\mu}(x)\right)^{\frac{1}{2}} m(g, x) f(g^{-1} \cdot x),$$ where $f$ comes from $L^2(\mathbb S, \mu, V).$ We assume that $T_g$ defines a unitary representation of $G.$ It is easily verified that $g \to T_g$ is a homomorphism if and only if the {{\it multiplier}} $m$ satisfies the cocycle identity

\begin{equation}\label{eqn:cocycle1}
m(g_1g_2, x) = m(g_1, x) m(g_2, g_1^{-1} \cdot x), \quad g_1, g_2 \in G, x \in \mathbb S.
\end{equation}

Next, set $\sigma(g) = m(g, x_0).$ Notice that,
\begin{equation}\label{eqn:sigma}
\sigma(hg) = \sigma(h) \sigma(g), \quad h \in H, g \in G.
\end{equation}
In particular, $\sigma$ restricted to $H$ is a homomorphism of $H$ into $\mathcal{U}(V),$ the group of unitary operators on $V$ and hence a unitary representation of $H$ as $m$ is Borel.

Next, we obtain a general form for the multiplier $m$ using $\sigma$ (this connects the representation of $H$ and the multiplier $m$). We start with the definition of a section.

If  $s : \mathbb S \to G$ is a Borel cross-section as in Remark \ref{cross-quasi}, then  every $g \in G$ admits a unique decomposition $g = s(x)h$ where $x = g \cdot x_0 $ and $h \in H.$ We shall need the following (Remark \ref{cross-quasi}):
\begin{lem}
There exists a Borel cross-section $s : \mathbb S \to G$ which is measurable.
\end{lem}
Notice that, for $x \in \mathbb S,$ $g \in G,$
\begin{equation}\label{eqn:form}
\sigma(s(x)^{-1}g) = m(s(x)^{-1}g, x_0) = m(s(x)^{-1}, x_0) m(g, x) = \sigma(s(x)^{-1}) m(g, x).
\end{equation}

Since, $$s(g^{-1} \cdot x)^{-1} g^{-1} s(x) \cdot x_0 = s(g^{-1} \cdot x)^{-1} g^{-1} \cdot x = x_0,$$ we have, $$s(g^{-1} \cdot x)^{-1} g^{-1} s(x) \in H.$$  So, $g^{-1} s(x) = s(g^{-1} \cdot x) h$ for some $h \in H.$ Write, $ s(x)^{-1} g = h^{-1} s(g^{-1} \cdot x)^{-1}.$ By \eqref{eqn:sigma} we get  $$ \sigma(s(x)^{-1} g) = \sigma(h^{-1}) \sigma(s(g^{-1} \cdot x)^{-1}), \quad g^{-1} s(x) = s(g^{-1} \cdot x) h.$$ Using the above and \eqref{eqn:form}, we get the general form for a multiplier as
\begin{equation}\label{eqn:general-form-cocyle}
m(g, x) = \sigma (s(x)^{-1})^{-1} \sigma(h^{-1}) \sigma(s(g^{-1} \cdot x)^{-1}), \quad g^{-1}s(x) = s(g^{-1} \cdot x) h
\end{equation}

\subsection{Induced representation}
Next, we describe the representation theoretic construction of an induced representation and show that multiplier representations and induced representations are the same. As above $\mathbb S = G/H$ and let $\mu$ be the $G-$invariant measure on $\mathbb S.$

Let $\sigma$ be a unitary representation of $H$ on a Hilbert space $\mathcal{H}_\sigma.$ We denote the inner product on $\mathcal{H}_\sigma$ by $\langle \cdot ,~ \cdot \rangle_\sigma.$ Let $\mathcal{F}_c$ be the vector space of functions taking values in $\mathcal{H}_\sigma$ defined by 
\[ \mathcal{F}_c = \{ F \in C(G, \mathcal{H}_\sigma):~F(gh) = \sigma(h^{-1})F(g)~{\rm{and}}~q({\rm{supp}}(F))~{\rm{is~compact}}~\},\] 
where $q :G \to G/H$ is the canonical projection. Notice that, if $F_1, F_2 \in \mathcal{F}_c,$ the function 
$\langle F_1(g), F_2(g) \rangle_{\sigma}$ depends only on the coset $q(h) = gH$ (as the representation $\sigma$ is unitary) and so is a function in $C_c(G/H).$


Define an inner product on $\mathcal{F}_c$ by 
\[ \langle F_1, F_2 \rangle = \int_{\mathbb{S}}~\left \langle F_1(x), F_2(x)\right \rangle_\sigma~d\mu(x).\] 

It is easy to verify that the above defines an inner product on $\mathcal{F}_c.$ Let $\mathcal{H}$ be the Hilbert space obtained by completing $\mathcal{F}_c$ with respect to this inner product. Notice that the space $\mathcal{F}_c$ is invariant under the operators $L_g, g \in G$, defined by $$L_g F(g_1) = \left (\frac{d(g_*\mu)}{d\mu}(g_1 \cdot x_0) \right)^{\frac{1}{2}} F(g^{-1} g_1)$$ and the inner product on $\mathcal{F}_c$ is preserved by these maps $L_g, g \in G$ (because $\mu$ is quasi $G$-invariant). It follows that $L_g$ for each $g \in G$, extends to an isometry on $\mathcal{H}.$ Since it is easy to see that $L_{g_1 g_2} = L_{g_1} L_{g_2}$ on $\mathcal{F}_c,$ each $L_g$ becomes a unitary operator on $\mathcal{H}.$ Moreover, the map $g \to L_g F$ from $G$ to $\mathcal{F}_c$ (for a fixed $F$) is continuous and since the operators $L_g$ are uniformly bounded, they are strongly continuous for all $F \in \mathcal{H}.$ Thus we have a unitary representation of $G$ called the induced representation, denoted by ${\rm{Ind}}_H^G(\sigma).$ We may also use $L_g$ to denote this representation when there is no scope for confusion.
\begin{ex}
Let $\sigma$ be the trivial representation of $H$ on $\mathbb C.$ Then, $\mathcal{F}_c$ is naturally identified with $C_c(G/H) = C_c(\mathbb S) $ and $\mathcal{F} = L^2(G/H) = L^2(\mathbb S, d\mu).$ If $G/H$ admits a $G$-invariant measure $\mu$, the representation ${\rm{Ind}}_H^G(\sigma)$ is the left regular representation of $G$ on $L^2(G/H, d\mu).$
\end{ex}

\begin{rem}
The elements in the space $\mathcal{H}$ are in fact functions that take values in $\mathcal{H}_\sigma$, see {{\it Remark 1}} in \cite[Chapter 6]{GF}.
\end{rem}
For $F \in \mathcal{H},$ define the $\mathcal{H}_\sigma$ valued function $\widetilde{F}$ on $\mathbb S$ by $\widetilde{F}(x) = F(s(x)),$ where $s: \mathbb S \to G$ is the measurable section chosen earlier. Then, it can be verified that the map $F \to \widetilde{F}$ is a unitary transformation from $\mathcal{H}$ onto the Hilbert space $\widetilde{\mathcal{H}}$ where $$ \widetilde{\mathcal{H}} = L^2(\mathbb S, \mathcal{H}_\sigma, d\mu) = \{ \widetilde{F} : \mathbb S \to \mathcal{H}_\sigma, \quad \int_\mathbb S \| \widetilde{F}(x)\|_\sigma^2 d\mu(x) < \infty \}.$$ Transferring the action of $G$ on $\mathcal{H}$ to $\widetilde{\mathcal{H}}$ we obtain the representation $\widetilde{L_g}$ of $G$ defined by
\begin{align}\label{eqn:lgtilde}
\widetilde{L_g}\widetilde{F}(x) &= \left (\frac{d(g_* \mu)}{d\mu}(x) \right)^{\frac{1}{2}}\sigma(h^{-1}) \widetilde{F}(g^{-1}s(x) \cdot x_0)\nonumber \\
&= \left (\frac{d(g_* \mu)}{d\mu}(x) \right)^{\frac{1}{2}} \sigma(h^{-1}) \widetilde{F}(g^{-1} \cdot x),
\end{align}
where $h \in H$ is determined by $g^{-1}s(x) = s(g^{-1}\cdot x)h.$ The representation $\widetilde{L_g}$ is, of course, unitarily equivalent to
${\rm{Ind}}_H^G(\sigma)$.
Recall that the representation ${\rm{Ind}}_H^G(\sigma)$ is  realized (as ${L}_g$) on the Hilbert space ${\mathcal{H}}.$

We now show the correspondence between multiplier representations and induced representations. For this consider the multiplier representation $T_g$ described earlier with the cocycle $m$ and its restriction $\sigma$. Note that $\sigma$ is defined on all of $G$ and when restricted to $H$ gives a unitray representation of $H$, see Equation \eqref{eqn:sigma}. 

Define the map $\mathcal{A}$ on $\widetilde{\mathcal{H}}$ by $\mathcal{A}f(x) = \sigma(s(x))^{-1}f(x).$ Clearly, $\mathcal{A}$ is an invertible transformation. Let us compute $\mathcal{A}T_g\mathcal{A}^{-1}.$ For $f \in \widetilde{\mathcal{H}}$ we have,
\begin{eqnarray*}
  \mathcal{A}T_g\mathcal{A}^{-1}f(x) &=& \sigma(s(x))^{-1} \left ( T_g \mathcal{A}^{-1} f \right )(x) \\
   &=& \sigma(s(x))^{-1} \left (\frac{d(g_* \mu)}{d\mu}(x) \right)^{\frac{1}{2}} m(g, x) (\mathcal{A}^{-1}f)(g^{-1} \cdot x) \\
   &=& \sigma(s(x))^{-1} \left (\frac{d(g_* \mu)}{d\mu}(x) \right)^{\frac{1}{2}} m(g, x) \sigma(s(g^{-1} \cdot x))f(g^{-1} \cdot x) \\
   &=& \left (\frac{d(g_* \mu)}{d\mu}(x) \right)^{\frac{1}{2}} \sigma(h^{-1}) f(g^{-1} \cdot x),
\end{eqnarray*} where we used \eqref{eqn:general-form-cocyle} in the last step. But the above equals $\widetilde{L_g}f(x)$ (see \eqref{eqn:lgtilde}) and hence $T_g$ is equivalent to ${\rm{Ind}}_H^G(\sigma)$ with the representation $\sigma$ from which the induction is done being the representation obtained by restricting the cocycle $m$ as mentioned above. 

Conversely, let $\sigma$ be a unitary representation of $H$ on a Hilbert space $\mathcal{H}_\sigma.$ We shall construct a multiplier $m$ so that the corresponding multiplier representation, denoted by $T_g^m$ is unitarily equivalent to ${\rm{Ind}}_H^G(\sigma).$ Notice that we can not define the multiplier $m$ using the form given in \eqref{eqn:general-form-cocyle} as $\sigma(s(x))$ does not make sense if $s(x) \notin H$ (which is the case if $x \neq x_0$). However, notice that $s(x)^{-1} g s(g^{-1} \cdot x) \in H,$ as $$s(x)^{-1} g s(g^{-1} \cdot x) \cdot x_0 = s(x)^{-1} g (g^{-1} \cdot x) =
s(x)^{-1} \cdot x = x_0.$$ Hence we can define  
\[m(g, x) = \sigma (s(x)^{-1} g s(g^{-1} \cdot x)).\]
It is easy to verify that $m$ satisfies the cocycle identity \eqref{eqn:cocycle1}. Define $T_g^m$ on $\widetilde{\mathcal{H}}$ by 
\[T_g^mf(x) = \left (\frac{d(g_* \mu)}{d\mu}(x) \right)^{\frac{1}{2}} m(g, x) f(g^{-1} \cdot x).\] 
Clearly, $g \to T_g^m$ is a unitary representation of $G.$ Now, following as above by defining $\widetilde{\sigma}(g) = m(g, x_0)$ we obtain a unitary representation $\widetilde{\sigma}$ of the subgroup $H$ and that $T_g^m$ is unitarly equivalent to ${\rm{Ind}}_H^G(\widetilde{\sigma}).$ Hence, it suffices to prove that $\sigma$ and $\widetilde{\sigma}$ are unitarily equivalent as {\textit{representations of}} $H.$ Now, for $h \in H,$ 
\[ \widetilde{\sigma}(h) = m(h, x_0) = \sigma(s(x_0)^{-1} h s(h^{-1} \cdot x_0)) = \sigma(s(x_0))^{-1} \sigma(h) \sigma(s(x_0))\] 
as $s(x_0) \in H$ and $\sigma$ is a representation of $H,$ which proves that $\sigma$ and $\widetilde{\sigma}$ are equivalent. So, it follows that $T_g^m$ and ${\rm{Ind}}_H^G(\sigma)$ are unitarily equivalent. Thus we have proved the following:
\begin{thm} \label{indmultrepn}
Let $G$ be a second countable locally compact group and $H$ be a closed subgroup of $G.$
\begin{enumerate}
\thmlist
\item Let $m$ be a multiplier defined on $G \times \mathbb S$ taking values in the group of unitary operators on a Hilbert space. Then the multiplier representation $T_g^m$ defined by 
$$T_g^m f(x) = \left (\frac{d(g_* \mu)}{d\mu}(x) \right)^{\frac{1}{2}} m(g, x) f(g^{-1} \cdot x),$$
is unitarily equivalent to ${\rm{Ind}}_H^G(\sigma)$ where the representation $\sigma$ of $H$ is given by $\sigma(h) = m(h, x_0).$
\item Conversely, let $\sigma$ be a unitary representation of $H$ and let $$m(g, x) = \sigma (s(x)^{-1} g s(g^{-1} \cdot x)).$$ Then $m$ is a multiplier and ${\rm{Ind}}_H^G(\sigma)$ is equivalent to the multiplier representation $T_g^m.$
\end{enumerate}
\end{thm}

A lot more on the inducing construction and the imprimitivity can be found in sections 1 and 2 of \cite{Kirillov}. In particular, see Theorem 4 and 5 of \cite{Kirillov}.

\section{Imprimitivity and homogeneous normal operators}

Let $\boldsymbol N$ be a commuting tuple of normal operators acting on a complex separable Hilbert space $\mathcal H$. Let $\mathcal C^*({\boldsymbol{N}})$ be the unital $C^*$-algebra generated by $(N_1, \ldots , N_d)$. This $C^*$-algebra is $*$- isomorphic to the algebra of continuous functions $C(S)$ for some compact Hausdorff space $S\subset \mathbb C^d$ via the Gelfand representation theorem. The joint spectrum $\sigma(\boldsymbol{N})$ of $(N_1, \ldots , N_d)$ is the set $S$. The continuous functional calculus for the commuting $d$-tuple $\boldsymbol N$ defined using the Gelfand transform provides a $*$-homomorphism $\rho_{\boldsymbol{N}}: C(\sigma(\boldsymbol N)) \to \mathcal L(\mathcal H)$ by setting $\rho_{\boldsymbol{N}}(f) = f(\boldsymbol N)$. On the other hand, if   $\rho: C(\sigma(\boldsymbol N)) \to \mathcal L(\mathcal H)$ is any $*$-homomorphism, then setting $N_i= \rho(z_i)$, where $z_i$ are the coordinate functions, we see that $(N_1, \ldots , N_d)$ is a commuting tuple of normal operators. The $*$-homomorphism $\rho_{\boldsymbol{N}}$ induced by $\boldsymbol N$ coincides with $\rho$.

Let $G$ be a second countable locally compact group and $\mathbb{S}$ be a locally compact topological space. We say that $\mathbb{S}$ is a $G$-space if there is a continuous map $G\times \mathbb{S} \to \mathbb{S}$ such that $(e,s) = s$, where $e$ is the identity in $G$ and for every pair $g, \hat{g}$ in $G$, $g\cdot (\hat{g}\cdot s) = (g \hat{g})\cdot s$, $s\in S$. It then follows that for every fixed $g\in G$, the map $g: \mathbb{S}\to \mathbb{S}$, $g(s):= g\cdot s$, $s\in \mathbb{S}$, is a continuous bijection of $\mathbb{S}$. Thus, $g=(g_1, \ldots , g_d)$ with $g_j:\mathbb{S}\to\mathbb C$, $1 \leqslant j \leqslant d$. These are the {\em coordinate functions} of the action $g:\mathbb{S} \to \mathbb C^d$ induced by $g\in G$.  

In what follows, we assume  that  $\sigma(\boldsymbol N)$ is a $G$-space and that the action of $g\in G$ is holomorphic in some open neighbourhood $\mathcal O_g$ of $\sigma(\boldsymbol N)$. This action of $G$ on $\sigma(\boldsymbol N)$ lifts to the commuting tuple  $\boldsymbol N$ by setting 
\begin{equation} \label{func}
g\cdot \boldsymbol N:= (g_1\cdot \boldsymbol N, \ldots , g_d\cdot \boldsymbol N). 
\end{equation}
\begin{defn} \label{homog} 
Let $\boldsymbol N = (N_1, \cdots , N_d)$ be a commuting tuple of normal operators such that the spectrum $\sigma(\boldsymbol N)$ is a $G$-invariant compact set. The commuting $d$-tuple $\boldsymbol N$ is said to be \textit{homogeneous} if there is a unitary representation $U$ of $G$ on $\mathcal H$ such that
\begin{equation}
   U_g^* \boldsymbol N U_g:=(U_g^* N_1 U_g, \ldots , U_g^* N_d U_g) = g \cdot \boldsymbol N,\, g\in G. 
\end{equation}
\end{defn}
To discuss induced representations for a locally compact second countable group $G$, Mackey introduced the notion of an imprimitivity which we recall below. 
For any locally compact Hausdorff space $\mathbb{S}$, we let $C_0(\mathbb{S})$ denote the $C^*$- algebra of continuous functions vanishing at $\infty$, that is, if $f\in C_0(\mathbb{S})$, and $\epsilon$ is any positive number, then there is a compact subset $K$ of $\mathbb{S}$ such that $|f(s)| <\epsilon$ for all $s \not \in K$. The action of the group $G$ lifts to an action on $C_0(\mathbb{S})$: $(g, f) \to g\cdot f$, where $(g\cdot f)(s) = (f\circ g^{-1})(s)$, $g\in G$, $s\in \mathbb{S}$. 

\begin{defn} \label{imp} Suppose that $U$ is a unitary representation of the group $G$ on some Hilbert space $\mathcal H$ and $\rho:C_0(\mathbb{S}) \to \mathcal L(\mathcal H)$ is a $*$-homomorphism of the $C^*$-algebra of continuous functions on $\mathbb{S}$ vanishing at infinity. The \textit{imprimitivity} introduced by Mackey is the requirement 
\begin{equation} \label{imprel} U_g \rho(f) U_g^* = \rho(g \cdot f),\, f\in C_0(\mathbb S), \tag{$\dagger$}\end{equation}
where $(g\cdot f)(s) = f (g^{-1} \cdot s)$, $s\in \mathbb S$. 
\end{defn}
If $(S,U, \rho)$ is an imprimitivity for some compact set $S$, then the $d$-tuple $(\rho(z_1), \ldots , \rho(z_d))$ of commuting normal operators is \textit{homogeneous} by definition, see \eqref{imprel} above, with $\sigma(\rho(z_1), \ldots , \rho(z_d))=S$. The other way round, the theorem below shows that if $\boldsymbol N$ is a $d$-tuple of \textit{homogeneous} normal operators  with associated representation $U$, then $(\sigma(\boldsymbol N),U, \rho_{\boldsymbol{N}})$ is an imprimitivity.
\begin{thm} \label{homognormal} Let $\boldsymbol N:=(N_1, \ldots , N_d)$ be a $d$-tuple of commuting normal operators defined on a complex separable Hilbert space $\mathcal H$. Assume that $\boldsymbol N$ is homogeneous under the action of a group $G$ with associated representation $U$. Then $(\sigma(\boldsymbol N), U, \rho_{\boldsymbol{N}})$ is an imprimitivity.
\end{thm} 
\begin{proof} By hypothesis, we have $U_g^* \boldsymbol{N} U_g = g\cdot \boldsymbol{N}$, $g\in G$, for some unitary representation $U$ of $G$. Taking adjoint of both sides, we get $U_g^* \boldsymbol{N}^* U_g = (g\cdot \boldsymbol{N} )^*$, where $\boldsymbol{N}^*=(N_1^*, \ldots , N_d^*)$ and similarly, $(g\cdot\boldsymbol{N})^* = ( (g_1\cdot \boldsymbol N)^*, \ldots , (g_d\cdot \boldsymbol{N})^*)$. Or, equivalently, \[(g\cdot\boldsymbol{N})^*=(\overbar{g_1} \cdot \boldsymbol{N}^* , \ldots , \overbar{g_d}\cdot \boldsymbol{N}^*),\]
where $\overbar{g_j}$,  denotes the complex conjugate of the coordinate functions $g_j$, $1\leq j \leq d$.  We adopt the convention: $(g\cdot\boldsymbol{N})^*= \overbar{g} \cdot \boldsymbol N^*$. 
Thus, if $\boldsymbol{N}$ is a commuting $d$- tuple of homogeneous normal operators, then we also have $U_g^* \boldsymbol{N}^* U_g = \overbar{g}\cdot (\boldsymbol{N}^*)$, $g\in G$.

The homogeneous $d$-tuple $\boldsymbol{N}$ defines a $*$-homomorphism $\rho_{\boldsymbol{N}}:C(\sigma(\boldsymbol{N})) \to \mathcal L(\mathcal H)$ and the homogeneity ensures that 
\begin{align} \label{1.1}
U_g^*(\rho(z_1), \ldots , \rho(z_d)) U_g&:= U_g^* \rho(\boldsymbol{z})U_g \nonumber \\
& = (g_1\cdot \boldsymbol{z}, \ldots , g_d\cdot \boldsymbol{z}) := \rho(g \cdot \boldsymbol{z}),\,\boldsymbol{z} \in \sigma(\boldsymbol{N}),\, g\in G.
\end{align} 
Similarly, the equality $U_g^* \boldsymbol{N}^* U_g = \overbar{g} \cdot (\boldsymbol{N}^*)$ becomes 
\begin{equation} \label{1.2}
U_g^*\rho(\overbar{\boldsymbol{z}}) U_g = \rho(\overbar{g} \cdot \overbar{\boldsymbol{z}}), g\in G. 
\end{equation}



Setting, $f_m(\boldsymbol{z})=z_1^{m_1} z_2^{m_2} \cdots z_d^{m_d}$, and $f_{mn}(z)= f_m \overline{f_n}$, we have that 
\begin{align*}
U_g^* \rho(f_m) U_g &= U_g^* N_1^m \cdots N_d^{m_d} U_g \\ &= \big ( (U_g^* N_1 U_g)^{m_1} (U_g^* N_2U_g)^{m_2} \cdots (U_g^*N_dU_g)^{m_d}\big )\\
& = ( U_g^* \rho(z_1) U_g )^{m_1} \cdots ( U_g^* \rho(z_d) U_g )^{m_d} \\ 
&= \rho(g_1\cdot \boldsymbol{z})^{m_1} \cdots \rho (g_d\cdot \boldsymbol{z})^{m_d}\\ 
&=\rho( (g_1 \cdot \boldsymbol{z})^{m_1} \cdots  (g_d \cdot \boldsymbol{z})^{m_d})\\
& =\rho(g \cdot f_m).
\end{align*}
Similarly, as before, 
\begin{align*}
U_g^* \rho(\overline{f_n}) U_g &= U_g^* {N_1^*}^{n_1} \ldots {N_1^*}^{n_d}  U_g \\ 
&= \big ( (U_g^* N_1^* U_g)^{n_1} (U_g^* N_2^* U_g)^{n_2} \cdots (U_g^* N_d^* U_g)^{n_d}\big )\\
& = ( U_g^* \rho(\overbar{z_1}) U_g )^{n_1} \cdots ( U_g^* \rho(\overbar{z_d}) U_g )^{n_d} \\ 
&=\rho(\overbar{g_1} \cdot \overbar{\boldsymbol{z}})^{n_1} \cdots  \rho(\overbar{g_d} \cdot \overbar{\boldsymbol{z}})^{n_d})\\ 
&=\rho( (\overbar{g_1} \cdot \overbar{\boldsymbol{z}})^{n_1} \cdots  (\overbar{g_d} \cdot \overbar{\boldsymbol{z}})^{n_d})\\ 
&=\rho(g \cdot \overbar{f_n})
\end{align*}
Finally, the computation below shows that $U_g^* f_{mn} U_g = \rho(g\cdot f_{mn})$:
\begin{align*}
    U_g^* \rho(f_{mn}) U_g &= U_g^* \rho(f_m) U_g U_g^* \rho(\overline{f_n}) U_g = \rho(g\cdot f_m) \rho(g\cdot \overline{f_n}) \\&= \rho\big ( (g\cdot f_m)(g\cdot \overline{f_n})\big ) = \rho(g \cdot f_{mn} ). 
\end{align*}
Now if $p$ is a polynomial of the form $p= \sum_{m,n} a_{mn} f_{mn}$, we clearly have 
$U_g^* \rho(p) U_g = \rho(g\cdot p)$. By the Stone-Weierstrass theorem, these polynomials are dense in $C(\sigma(\boldsymbol{N}))$. Thus, 
\[ U_g^* \rho(f) U_g = \rho(g\cdot f), f\in C(\sigma(\boldsymbol{N})).\qedhere\]
\end{proof} 

Next, we show that every imprimitivity $(\mathbb{S},U,\rho)$ based on a locally compact $G$- space $\mathbb{S}$ defines a homogeneous commuting normal tuple $\boldsymbol N$ with spectrum $\overbar{\mathbb{S}}$, where $\overbar{\mathbb{S}}$ is the closure of $\mathbb{S}$. For this, it would be useful to describe the imprimitivity in an equivalent form involving a \textit{spectral measure} on $\mathbb{S}$.

\begin{defn} \label{ProjMeas}
Let $G$ be a second countable locally compact group and $\mathbb{S}$ be a locally compact $G$- space. 
Suppose that  $U$ is a unitary representation of $G$ on a Hilbert space $\mathcal{H}$ and $P$ is a regular $\mathcal{H}$-projection-valued measure on $\mathbb{S}$. Then $(\mathbb{S},U,P)$ is said to be a system of imprimitivity if
\begin{equation}\label{dagger}
U(g) P(E) U(g)^{-1}=P(g\cdot E)
\end{equation}
for all  $g \in G$ and every Borel subset $E$ of $\mathbb{S}$. 
\end{defn} 
It is proved in \cite[p. 180]{GF} that the Definitions \ref{imp} and \ref{ProjMeas} are equivalent. For the proof of Proposition \ref{impcomp} given below, we will find it convenient to use the formulation of imprimitivity  in Definition \ref{ProjMeas}. 


The \textit{Imprimitivity theorem} due to Mackey says, among other things, that if $(\mathbb S, U, \rho)$ is a transitive system of imprimitivity, then $U$ is unitarily equivalent to the induced representation $\operatorname{Ind}_H^G(\sigma)$ for some unitary representation $\sigma$ of $H$, where the homogeneous space $\mathbb S$ is assumed to be of the form $G/H$, $G$ locally compact, $H\subset G$ closed. 

Applying the Hahn-Hellinger theorem to the spectral measure $P$,  writing it in the canonical form Theorem \ref{Hahn}, and imposing the  imprimitivity condition of Equation \eqref{dagger} from Definition \ref{ProjMeas} gives a canonical form of the  imprimitivity $(\mathbb S, U, P)$ as described below, see \cite[Proposition 4.1 and Theorem 4.3(b)]{VSS} (and also \cite[Theorem 6.12]{VSV}, \cite[Theorem 4]{Kirillov}). 

\begin{thm} [Imprimitivity theorem]
   Let $\mathbb S=G/H$ be a homogeneous $G$-space 
   and $\mu$ be a quasi-invariant measure on $\mathbb S$ (there is always one such uniquely determined modulo mutual absolute equivalence). 
   Assume that $(\mathbb S, U, \rho)$ is an imprimitivity acting on some separable complex Hilbert space $\mathcal H$. Then there is 
   a Hilbert space $V$ such that the Hilbert space $\mathcal{H}$ is isometric to $\mathcal{H}=L^2(\mathbb S, \mu, V)$, where  
   \begin{enumerate} 
   \item[(i)] $\mu$ is a quasi-invariant measure on $\mathbb{S}$ determined uniquely modulo equivalence, 
   \item[(ii)]  the representation $\rho$ is of the form $\rho(f) = M_f$, $f\in C_0(\mathbb S)$, and 
   \item[(iii)] the representation $U$ is of the form
\[
(U(g) f)(s)= \sqrt{\frac{d \mu(g \cdot s)}{d \mu(s)}} \sigma(h) f(g \cdot s ). 
\]
Here $h\in H$ is determined from the relation $g\, p(g^{-1} \cdot s) = p (s) h$, $s\in \mathbb S$, where $p:G/H \to G$ is a Borel cross-section.
\end{enumerate}

\end{thm}

We let $\mathcal{B}$ denote the Borel subsets of $\mathbb{C}^d$. The Borel $\sigma$- algebra of $X\subset \mathbb{C}^d$ then consists of all subsets of $X$ of the form $X\cap B$, $B\in \mathcal B$. If $E$ is a subset of $X$, then we say that a measure $\mu$ \textit{lives} on $E$ if $\mu (X\setminus E) = 0$. Two measures $\mu$ and $\nu$ are said to be mutually singular if they live on two disjoint sets.  

We have gathered all the tools from spectral theory of commuting tuples of normal operators to prove that every imprimitivity $(\mathbb{S}, U, \rho)$ based on a locally compact transitive $G$- space $\mathbb{S}$ is the restriction of an imprimitivity based on $\overbar{\mathbb{S}}$.  

\begin{thm} \label{impcomp} Suppose that $\mathbb{S}$ is a locally compact transitive $G$- space and the action of $G$ extends to  $\overbar{\mathbb{S}}$, the closure of $\mathbb{S}$ with $g\cdot \partial \mathbb S \subseteq  \partial \mathbb S$. 
\begin{enumerate} 
\item If $(\mathbb{S},U,P)$ is an imprimitivity, then there exists a unique spectral measure $\hat{P}$ defined on the Borel $\sigma$- algebra $\mathcal B$ of $\overbar{\mathbb{S}}$ satisfying the imprimitivity condition \eqref{dagger} with $\hat{P}(E) = P(E)$ for every Borel subset $E$ of $\mathbb{S}$. Moreover, $\operatorname{supp}(P) = \overbar{\mathbb{S}}$. 
\item If $(\mathbb{S},U,P)$ is an imprimitivity, then 
it defines uniquely a homogeneous commuting tuple of normal operators $\boldsymbol{N}$ such that $\sigma(\boldsymbol{N})= \operatorname{supp}(\hat{P}) = \overbar{\mathbb{S}}$, where $\hat{P}$ is the spectral measure of $\boldsymbol{N}$.
\end{enumerate}
\end{thm}
\begin{proof} The imprimitivity on $\mathbb{S}$ comes equipped with a spectral measure $P$ defined on the Borel subsets of $\mathbb{S}$. Define $\hat{P}$ on any Borel subset $E$ of $\overbar{\mathbb S}$ by setting 
\[\hat{P}(E)  = \begin{cases} P(E \cap \mathbb S) & \operatorname{if} E \cap \mathbb S \ne \emptyset, \\
0 & \operatorname{if} E\subseteq \partial \mathbb S. \end{cases}
\]
Thus, $\hat{P}$ is an extension of the spectral measure $P$ to the Borel $\sigma$- algebra on $\overbar{\mathbb{S}}$ with $\hat{P}(\partial \mathbb{S})=0$. Clearly, these properties determine $\hat{P}$ uniquely.  

Pick $s \in E$ and use Urysohn's Lemma to find a continuous function $f$ defined on $\overbar{\mathbb{S}}$ with $f\left(s\right)=1$, $f=0$ on $E^c$. If we had $P(E)=0$, then $\int_\mathbb{S} f dP =0$. This is not possible since the $*$- homomorphism $\rho$ of $C(\mathbb{S})$ defined by $\rho(f) = \int_{\mathbb{S}} f dP$ cannot be the zero operator. Hence, $P(E) \ne 0$ whenever $E$ is open in $\overbar{\mathbb{S}}$.  Consequently, $\operatorname{supp}(P)=\overbar{\mathbb{S}}$. 


    Starting with the imprimitivity, $(\mathbb{S}, U,P)$, we define the spectral measure $\hat{P}$ on $\overbar{\mathbb{S}}$ as in the Theorem. Since the action of $G$ leaves $\mathbb{S}$ and $\partial\mathbb{S}$ invariant, it follows that  $(\overbar{\mathbb{S}}, U, \hat{P})$ is an imprimitivity. Or, equivalently, $(\overbar{\mathbb{S}}, U, \rho)$ is an imprimitivity, where $\rho(f) = \int_{\overbar{\mathbb{S}}} f d\hat{P}$ is a $*$-homomorphsim of $C(\overbar{\mathbb{S}})$. By the imprimitivity theorem, the $*$-homomorphsim $\rho$ can be realized on the Hilbert space $L^2(\mathbb S, \mu, V)$ in the form $\rho(f) = M_f$, $f\in C(\overbar{\mathbb{S}})$. The $d$- tuple  $(\rho(z_1), \ldots , \rho(z_d))$ of multiplication by the coordinate functions is  homogeneous. To verify this claim, we have to find a unitary representation $\Gamma$ of $G$ and show that 
    \[\Gamma(g)^* \rho(z_k) \Gamma(g) = g\cdot {z_k}, \, k=1, \ldots , d,\]
where $g\cdot (\boldsymbol{z} \mapsto z_k) = z_k (g \cdot \boldsymbol{z}) = g_k(\boldsymbol{z})$. (Recall that $g(\boldsymbol{z}) = (g_1(\boldsymbol{z}), \ldots , g_d(\boldsymbol{z}))$.) Evidently, 
   this is the imprimitivity condition for the particular choice of the function $z_k:=\boldsymbol{z} \mapsto z_k$. Hence choosing $\Gamma=U$, where $U$ is the unitary from the Imprimitivity theorem, we are done. 
Moreover, the spectrum of this $d$- tuple is $\overbar{\mathbb{S}}$ as shown in the first half of the proof. 
\end{proof}

\begin{defn} \label{mrepn} Let $G$ be a second countable locally compact group and $\mathbb{S}$ be a locally compact $G$- space. We say that a  representation $U$ of $G$ acting on a Hilbert space $\mathcal{H}$ consisting of functions (equivalence classes modulo zero sets) is a \textit{multiplier representation} if there is a concrete realisation of $U$ as follows:
\[ (U_g f)(z)=(m(g^{-1}, z))^{-1} f\left(g^{-1} z\right), \,\,z \in \mathbb{S}, f \in \mathcal{H}, g \in G.\] 
\end{defn}
Let $S\subset \mathbb C^d$ be a compact $G$- space. Assume that $S$ decomposes into finitely many $G$ orbits such that $S= S_0\cup S_1 \cdots \cup S_r$. Thus, $G$ acts transitively on each $S_j$, and therefore $S_j \cong G/H_j$ for some closed subgroup $H_j$ of $G$, $0\leqslant j \leqslant r$. There exists a unique quasi-invariant measure on each of these orbits modulo mutual absolute equivalence. We choose and fix one such quasi-invariant measure on $S_j$, say, $\mu_j$.  
\begin{lem}
Let $\mu$ be a Borel measure on $S$ that is quasi-invariant with respect to the $G$- action on $S$. Then $\mu\cong \sum_{j=0}^r \mu_j$.
\end{lem}
\begin{proof}
    The proof follows from the uniqueness of the quasi-invariant measure $\mu_j$ on $S_j$, $0\leqslant j \leqslant r$.  
\end{proof}

\begin{thm} \label{multR}
    Let $\boldsymbol{N}$ be a homogeneous $d$- tuple of commuting normal operators acting on some Hilbert space $\mathcal{H}$ and let $S:=\sigma(\boldsymbol{N})$ be the spectrum of $\boldsymbol{N}$. Assume that $S$ is a $G$- space and that $S=\cupdot_{j=0}^r S_j$, where each $S_j$ is a $G$- orbit. Then the imprimitivity $(S, U, \rho_{\boldsymbol{N}})$ induced by $\boldsymbol{N}$ is equivalent to the imprimitivity $(S, \pi_\mu, \hat{U})$, i.e., there is a unitary \[\Gamma: \mathcal{H} \to \oplus L^2(E_n, \mu :\mathcal H_n)\] such that $\Gamma \rho(f) \Gamma^* =\pi_\mu(f)$, $f \in C(S)$ and $\Gamma {U} \Gamma^*= \hat{U}$ is a multiplier representation.  
    \end{thm}
\begin{proof} Applying the Hahn-Hellinger theorem to the $*$- representation $\rho$ of $C(S)$, we find a disjoint sequence $E_n$, $n\geqslant 1$, of Borel subsets of $S$ such that (i) $S=\bigcupdot_{n=1}^\infty E_n$,   
(ii) $\mathcal H\cong \oplus L^2(E_n, \boldsymbol{\mu} :\mathcal H_n)$, 
and (iii) $\rho(f) \cong \oplus_{n\geqslant 1}\pi^{(n)}(f)$, $f\in C(S)$, where $\pi^{(n)}=\pi_{\left .\mu\right|_{E_n}}^n$ is the canonical $*$- representation of $C(S)$. 

Since $\rho$ and $g\cdot \rho$ are unitarily equivalent, it follows from the uniqueness assertion in the Hahn-Hellinger theorem that $\mu\cong \mu\circ g^{-1}$ and $\mu(E_n \triangle g\cdot E_n) =0$, $g\in G$. Since $\mu$ is quasi-invariant it lives on a union of $G$- orbits, namely, $S_0 \cup S_1 \cup \cdots \cup S_r$.



Let $\hat{n}$ be the multiplicity function defined on $S$, i.e., $\hat{n}: S\to  \mathbb{N}$, $\hat{n}_{|E_n} = n$. By the Hahn-Hellinger theorem, $\hat{n} \circ g = \hat{n}$ for all $g\in G$. It follows that $\hat{n}$ is a constant function on each orbit of $G$. Consequently, $S_j \subseteq E_{n_j}$ for some $n_j$. Let $\boldsymbol{S}_j = \cupdot \{S_k: S_k \subseteq E_{n_j}\}$. Hence $S= \cup \{\boldsymbol{S}_j: j \in \boldsymbol{F}\subseteq \{0,1, \ldots, r\}\}$. 
We also see by the uniqueness portion of the Hahn-Hellinger theorem, that $\mu_{n_j} (E_{n_j} \triangle \boldsymbol{S}_j)=0$. Thus, $\mu_{| \boldsymbol{S}_j} = \mu_{j_1} + \cdots + \mu_{j_k}$, where $\mu_{j_k}$ is the quasi-invariant measure that lives on $S_{j_k}$ with $S_{j_k} \subseteq E_{n_j}$. It follows that modulo unitary equivalence via an  unitary $\Gamma: \mathcal H \to \oplus_{n=1}^\infty L^2(E_n, \mu, \mathcal{H}_n)$ we have  
\begin{align*}
\mathcal{H} \cong \bigoplus_{j\in \boldsymbol{F}} L^2(E_{n_j}\cap\boldsymbol{S}_j, \mu_{|\boldsymbol{S}_j}, \mathcal{H}_{n_j}) =  \bigoplus_{j=0}^r L^2(S_j, \mu_{|S_j}, \mathcal{H}_{n_j}).
\end{align*}
Let $\hat{\rho}= \Gamma \rho \Gamma^*$ and $\hat{U} = \Gamma U \Gamma^*$ and note that the pair $(\hat{\rho}, \hat{U})$ satisfy the imprimitivity condition. For each $j=0,1,\ldots ,r$, define the unitary representation $V_g^j$ on $L^2(S_j,\mu_j, \mathcal H_{n_j})$ by setting 
\[V_g^j \psi (x) = \big ( \tfrac{d(\mu_j \circ g^{-1})}{d\mu_j}(x)\big )^{1/2} \big (\psi\circ g^{-1} \big ) (x),\,x\in S_j,\, \psi \in L^2(S_j, \mu_j, \mathcal{H}_{n_j}).\]
Let $V_g=\oplus_{j=0}^r  V_g^j$. By construction, $(S, \hat{\rho}, V_g)$ is an imprimitivity. Hence 
\[\hat{U}_g^* \hat{\rho} \hat{U}_g = \hat{\rho}(g\cdot f) = V_g^* \hat{\rho} V_g,\, g\in G.\]
Hence the unitary operators $\hat{U}_g V_g^*$, $g\in G$, commute with the set of normal operators $\{\hat{\rho}(f): f\in C(S)\}$. It follows that the reducing subspaces $L^2(S_j, \mu_j, \mathcal{H}_{n_j})$ of the operators $\hat{\rho}(f)$ ($f\in C(S)$)  are also invariant under the unitary operators $ \hat{U}_g V_g^*$  ($g\in G$). Consequently, these reducing subspaces are also left invariant under $\hat{U}_g$. Therefore, $\hat{U}_g$ has the form $\oplus_{j=0}^r \hat{U}^j_g$, where $\hat{U}^j_g$ is the restriction of $\hat{U}_g$ to the subspace $L^2(S_j, \mu_j, \mathcal{H}_{n_j})$. 

Recall that $\hat{U}_gV_g^*$ must be a multiplication by a Borel function $c_g(\cdot)$ defined on $S$ taking values in the unitary operators acting on $\oplus_{j=0}^r \mathcal{H}_{n_j}$. Furthermore, since $L^2(S_j, \mu_j, \mathcal{H}_{n_j})$ are reducing subspaces for $\hat{U}_gV_g^*$, we conclude that
$c_g = \oplus_{j=0}^r c_g^j$. Hence 
$\hat{U}^j_g{V^j}_g^*= \oplus_{j=0}^r c^j_g$. The representation $V^j_g$ is a multiplier representation by construction with multiplier $\big (\tfrac{d (g\cdot \mu_j)}{d\mu_j}\big )^{1/2}$, we have that 
\begin{align*} \hat{U}_g^j f(x) &= c_g^j(x) (V_g f)(x)\\ 
& = \big ( \tfrac{d(\mu_j \circ g^{-1})}{d\mu_j}(x)\big )^{1/2}  c^j_g(x) (g\cdot f)(x), \end{align*}
$x\in S_j$ and $f\in L^2(S_j, \mu_j, \mathcal{H}_{n_j})$. Since $\hat{U}^j_g$ is a homomorphism, it follows that $( \tfrac{d(\mu_j \circ g^{-1})}{d\mu})^{1/2} c^j_g$ is a cocycle. Moreover, $\big (\tfrac{d (g\cdot \mu_j)}{d\mu_j}\big )^{1/2}$ is evidently a cocycle and therefore $c^j_g$ is a cocyle taking values in the group of unitary operators acting on $\mathcal{H}_{n_j}$. 
\end{proof}
We can assert more than what is claimed in the theorem. This is the the corollary below, its proof is apparent from the proof of the theorem. 
\begin{cor} \label{maincor}
Let $\boldsymbol{N}$ be a homogeneous $d$- tuple of commuting normal operators acting on some Hilbert space $\mathcal{H}$ and let $S:=\sigma(\boldsymbol{N})$ be the spectrum of $\boldsymbol{N}$. Assume that $S=\cupdot_{j=0}^r S_j$, where each $S_j$ is a $G$- orbit and is not necessarily compact. Then there exist quasi-invariant measures $\mu_j$ living on $S_j$ such that $\boldsymbol{N}$ is unitarily equivalent to the direct sum of $\boldsymbol{M}^{(j)}$ of the multiplication by the coordinate functions acting on the Hilbert space $L^2(S_j,\mu_j, \mathcal{H}_{n_j})$, $\dim(\mathcal{H}_{n_j})=n_j$, $0 \leqslant j \leqslant r$. 
\end{cor} 
We point out that $\mathcal{H}_{n_j}$ may be isomorphic to $\mathcal{H}_{n_k}$ even if $j\ne k$.  
\begin{rem}
    Let $\boldsymbol{N}$ be a homogeneous $d$- tuple. Theorem \ref{multR} ensures that we may take, modulo unitary equivalence and without loss of generality, the unitary representation $U$  intertwining the commuting $d$- tuple $(N_1, \ldots , N_d)$ with $(g_1(\boldsymbol{N}), \ldots , g_d(\boldsymbol{N}))$ to be a multiplier representation. 
    Conversely, assume that the $d$- tuple $\boldsymbol{N}$ of normal operators is realized as multiplication $M_k$, $k=1, \ldots , d$, by coordinate functions $z_k$ on some $L^2(S,\mu,V)$. If there is a multiplier representation $U$ on $L^2(S,\mu,V)$ of the form $f\mapsto c(g,\boldsymbol{z}) f(g\cdot \boldsymbol{z})$, then
\begin{align*}
    \big (M_k U(g) f\big )(\boldsymbol{z}) &= M_k c(g,\boldsymbol{z}) f(g\cdot \boldsymbol{z})\\  
    &= g_k (g\cdot \boldsymbol{z}) c(g, \boldsymbol{z}) f(g\cdot \boldsymbol{z})\\
    &= \big ( U(g) M_{g_k} f\big )(\boldsymbol{z}),\, 1\leqslant k \leqslant d.
\end{align*}
    Thus, $U$ intertwines the commuting $d$- tuple $\boldsymbol{N}$ with $g\cdot \boldsymbol{N}$. Hence $\boldsymbol{N}$ is homogeneous. 
\end{rem}

\section{Examples}
The M\"{o}bius group acts transitively on the open unit disc $\mathbb{D}$ and the unit circle $\mathbb{T}$. 
However although the M\"{o}bius group acts on the closed disc $\overbar{\mathbb{D}}$, this action is no longer transitive. Indeed, $\overbar{\mathbb{D}}$  is the disjoint union of two orbits, namely, $\mathbb{D}$ and $\mathbb{T}$. When the action is transitive, the canonical form of the imprimitivity is described by Mackey. Hence if $N$ is a homogeneous normal operator with $\sigma(N) = \mathbb{T}$, then it must be a $n$- fold direct sum of the operator of multiplication by $\alpha$, $\alpha\in \mathbb{T}$ on $L^2(\mathbb{T}, d\theta)$, where $d\theta$ is the arc length measure. However, it is not obvious what are the homogeneous normal operators $N$ with $\sigma(N) = \overbar{\mathbb{D}}$. These are described explicitly in a forthcoming paper of the first named author with A. Kor\'{a}nyi, see also \cite[Theorem 6.6]{survey}.    

\subsection{The case of a product domain} 
In this section, we describe commuting pairs $\boldsymbol{N}$ of homogeneous normal operators with $\sigma(\boldsymbol{N})=\overbar{\mathbb{D}}\times\overbar{\mathbb{D}}$.
The subset $\mathbb D^2:=\mathbb D \times \mathbb D$ of $\mathbb{C}^2$ is a $G$- space, where
$G$ is the subgroup 
\[\{\phi:= (\phi_1, \phi_2) \mid \phi_k(z_k) = \beta_k \tfrac{z_k - \alpha_k}{1-\overbar{\alpha}_k z_k}, \beta_k\in \mathbb T, \alpha_k \in \mathbb D, k=1,2\}\]
of the group of bi-holomorphic automorphisms of $\mathbb D \times \mathbb D$. The automorphism $\phi$ extends to an automorphism of $\overbar{\mathbb{D}}\times\overbar{\mathbb{D}}$ with $\phi(\partial \mathbb D^2) \subseteq \partial \mathbb D^2$. 
To identify homogeneous (under the $G$- action) pairs of commuting normal operators, we first note that the spectrum of such a pair must be a $G$- invariant compact subset of $\mathbb C^2$. To find these, note that the orbit through a point $(z_1, z_2) \in \mathbb T\times \mathbb D$ is $\mathbb T \times \mathbb D$, similarly,   $\mathbb D \times \mathbb T$ is also a $G$- orbit. If $(z_1, z_2) \in \mathbb T\times \mathbb T$, the  $G$- orbit is $\mathbb T \times \mathbb T$. These are all the $G$- oribits in the boundary of $\mathbb{D}\times \mathbb{D}$.  Closure of these obits gives us compact sets that are G-invariant. Moreover, if $(z_1, z_2)$ is in $\mathbb D^2$, then the $G$- orbit through this point is $\mathbb D^2$. Thus, all the compact $G$- invariant subset of $\mathbb C^2$ are 
\[\overbar{\mathbb D} \times \overbar{\mathbb D},\,\, \mathbb T \times \overbar{\mathbb D},\,\, \overbar{\mathbb D} \times \mathbb T,\,\, \mathbb T \times \mathbb T.\]

Among these, the group $G$ acts transitively only on $\mathbb{T} \times \mathbb{T}$. Consequently, imprimitivities based on $\mathbb{T} \times \mathbb{T}$, or equivalently, associated pairs $\boldsymbol{N}$ of homogeneous normal operators, are described by Mackey's theorem. The remaining three cases can be dealt with using Corollary \ref{maincor}, however, we provide an elementary analysis below.  

If we consider a commuting pair of homogeneous normal operators $\boldsymbol{N}$ with $\sigma_{\boldsymbol{N}} = \overbar{\mathbb D} \times \overbar{\mathbb D}$, then it must be unitarily equivalent to the pair of multiplication operators $\boldsymbol{M}= (M_1, M_2)$ acting on $L^2(\overbar{\mathbb D} \times \overbar{\mathbb D}, \mu, \mathcal H_n)$,  where $\mu$ is quasi-invariant with respect to the group $G$ and $\dim \mathcal{H}=n$. The restriction of the measure $\mu$ to the transitive $G$- space $\mathbb D \times \mathbb D$, $\mathbb D\times \mathbb T$, $\mathbb T \times \mathbb D$ and $\mathbb T \times \mathbb T$ is uniquely determined since the group acts on these transitively. These are the measures: $\mu_1:=dA \times dA$, $\mu_2:=dA \times d\theta$, $\mu_3=d\theta \times dA$ and $\mu_4:=d\theta \times d\theta$, respectively. (Here, $dA$ and $d\theta$ denote the area and the arc length measure, respectively.) Evidently, $\mu = \mu_1 + \mu_2 +\mu_3 + \mu_4$. Moreover, $\mu_i$, $1\leqslant i \leqslant 4$, are mutually singular. Consequently, $L^2(\overbar{\mathbb D} \times \overbar{\mathbb D}, \mu, \mathcal H_n)$ must be a direct sum of the form 
\[L^2(\mathbb D \times \mathbb D, \mu_1, \mathcal H_{n_1}) \oplus L^2(\mathbb D \times \mathbb T, \mu_2, \mathcal H_{n_2}) \oplus L^2(\mathbb T \times \mathbb D, \mu_3, \mathcal H_{n_3})  \oplus L^2(\mathbb T \times \mathbb T, \mu_4, \mathcal H_{n_4}) , \]
where $n=n_1+n_2+n_3+n_4$. These are reducing subspaces of the homogeneous pair of operators $(M_1,M_2)$ acting on $L^2(\overbar{\mathbb D} \times \overbar{\mathbb D}, \mu, \mathcal H_n)$. We have therefore proved that any pair $\boldsymbol{N}$ of homogeneous normal operators with $\sigma(\boldsymbol{N})= \overbar{\mathbb D} \times \overbar{\mathbb D}$ must be unitarily equivalent to the direct sum of the pair $(M_1,M_2)$ of multiplication operators acting on the Hilbert space 
$\bigoplus_{i=1}^4 L^2 (X_i, \mu_i, n_i)$,
where 
$X_1=\mathbb D \times \mathbb D$, $X_2=\mathbb D \times \mathbb T$, $X_3=\mathbb T \times \mathbb D$, $X_4=\mathbb T \times \mathbb T$. 

The two cases where the spectrum of $\boldsymbol{N}$ is either $\mathbb {T}\times \overbar{\mathbb{D}}$ or $\overbar{\mathbb{D}}\times \mathbb{T}$ are described exactly in the same manner. 

\subsection{The case of a bounded symmetric domain} Let $\mathcal D\subset \mathbb C^d$ be a bounded symmetric domain and $G$ be its bi-holomorphic automorphism group. We note that the $G$ action on $\mathcal D$ extends to some open neighbourhood of $\mathcal D \cupdot \partial \mathcal D$.  The $G$ action is transitive on $\mathcal D$. The topological boundary $\partial \mathcal D$ splits into $r$ $G$-orbits, say, $S_0,\ldots , S_{r-1}$, where $r$ is the rank of $\mathcal D$. Also, setting 
\[\mathcal{S}_j = S_0 \cupdot S_1 \cupdot \cdots \cupdot S_j,\]
we note that $\mathcal{S}_j= \overbar{S_j}$, $0\leqslant j \leqslant r-1$. 
Here, $S_0$ is the Silov boundary and is necessarily compact.

The only possibilities for the spectrum of a $d$- tuple of commuting  homogeneous normal operators are compact $G$- invariant subsets of $\mathbb C^d$. These are the  sets: $\overbar{\mathcal{D}}$ and  $\mathcal{S}_j$, $0\leqslant j \leqslant r-1$. Fixing the spectrum of a $d$- tuple of commuting homogeneous normal operators to be one of these sets, we see that the hypothesis of Corollary \ref{maincor} are met. Therefore, we have a complete description of the commuting $d$- tuples of homogeneous normal operators, or equivalently, all the imprimitivities based on the $G$- spaces $\mathcal{S}_j$, $0\leqslant j \leqslant r-1$ and $\overbar{\mathcal{D}}$.


\section{Open problems}

Let $\mathbb{S}$ be a bounded open connected subset of $\mathbb{ C}^d$. Following the well established tradition in Operator theory, we propose to study homomorphisms of the ``disc'' algebra $\mathcal A(\mathbb{S}) \subset C(\overbar{\mathbb{S}})$, where $\mathcal A(\mathbb{S})$ consists of all those complex valued functions $f$ such that there is an open set $U$ containing $\overbar{\mathbb{S}}$ and that $f$ is holomorphic on $U$. Now, pick any algebra homomorphism $\varrho:\mathcal A(\mathbb{S}) \to \mathcal{L}(\mathcal{H})$ and a unitary representation of a locally compact second countable group $G$ on $\mathcal{H}$ and say that $(\mathbb{S}, \varrho, U)$ is a ``holomorphic imprimitivity" if 
\[U(g)^* \varrho(f) U(g) = \varrho(g\cdot f), \, \, g\in G\,\,f\in \mathcal{A}(\mathbb{S}).\]
A homomorphism $\varrho$ as above clearly defines a $d$- tuple of commuting  bounded linear operators $(T_1:=\varrho(z_1), \ldots , T_d:= \varrho(z_d))$ that is \textit{homogeneous}, namely, 
\[U(g)^* ( T_1 , \ldots , T_d) U(g) = ( g_1(T_1), \ldots ,g_d(T_d) ), \, \, g\in G.\]
Here we must assume that the joint spectrum (we take it to be the one defined in the sense of Taylor) of $\boldsymbol{T}$ is contained in $\overbar{\mathbb{S}}$.   

Conversely, starting with a $d$- tuple $\boldsymbol{T}$ of homogeneous commuting  bounded linear operators with $\sigma(\boldsymbol{T}) \subseteq \overbar{\mathbb{S}}$, a homomorphism $\varrho$ is defined by setting $\varrho_{\boldsymbol{T}}(f) = f(\boldsymbol{T})$ via the holomorphic functional calculus. 

A detailed study of the class of algebra homomorphisms $\varrho:\mathcal{A}(\mathbb{S}) \to \mathcal{L}(\mathcal{H})$ that are restrictions of $*$- homomorphisms $\hat{\rho}: B(\mathbb{S}) \to \mathcal{L}(\mathcal{K})$, where $\mathcal{K} \supset \mathcal{H}$ and $\hat{\rho}(f) (\mathcal{H}) \subseteq \mathcal{H}$ has been very fruitful. In this case the commuting $d$- tuple of operators $(T_1, \ldots , T_d)$ is called subnormal. 

 When $\mathbb{S}$ is only locally compact, as pointed out earlier, the $*$- representation $\hat{\rho}$ extends to a $*$- representation of the algebra $B(\mathbb{S})$ of bounded Borel functions on $\mathbb{S}$ and  
 $\mathcal{A}(\mathbb{S})\subset B(\mathbb{S})$. Thus, we can speak of $\hat{\rho}(f)$ for $f\in \mathcal{A}(\mathbb{S})$.
 
The first question in the context of this paper is the following. Suppose that $(\mathbb{S}, \varrho, U)$ is a holomorphic imprimitivity based on $\mathcal{H}$, and  there exists a $*$- homomorphism $\hat{\rho}: C_0(\mathbb{S})\to \mathcal{L}(\mathcal{K})$, $\mathcal{K}\supset \mathcal{H}$, such that (a) $\mathcal{H}$ is invariant for $\hat{\rho}(f)$, $f\in \mathcal{A}(\mathbb{S})$, and (b) $\hat{\rho}_{|\mathcal{H}}(f) = \rho(f)$, $f\in \mathcal{A}(\mathbb{S})$.
Then does it follow that there is a unitary representation $\hat{U}$ of $G$ on $\mathcal{K}$ such that $(\mathbb{S}, \hat{\rho}, \hat{U})$ is an imprimitivity?

Let $(\mathbb{S},\hat{\rho}, U)$ be an imprimitivity. The second question asks for a description of all the simultaneous invariant subspaces of $\hat{\rho}$ and $\hat{U}$. 

The first question has an affirmative answer when $\mathbb{S}$ is either $\overbar{\mathbb{D}}$ or $\mathbb{T}$.
This is in a forthcoming paper of the first named author jointly with A. K\'{o}ranyi. 
In the same paper, although, there are some partial results describing the simultaneous invariant subspaces in the case of  $\overbar{\mathbb{D}}$ and $\mathbb{T}$, the answer to the question of simultaneous invariant subspaces is  far from complete even in this very simple case. 

When the answer to the second question is affirmative, then restricting to the simultaneous invariant subspace we obtain homogeneous subnormal operators. If the answer to the first question is also affirmative, which is very likely, and if we have a description of all the simultaneous invariant subspaces, then all the homogeneous subnormal operators would be the restriction of homogeneous normal operators (or imprimitivities) to these subspaces. 


\end{document}